\documentclass{amsart}
\usepackage{amsmath, amssymb, mathrsfs, verbatim, multirow}
\usepackage[all]{xy}

\newtheorem{Teo}{Theorem}[section]
\newtheorem{Prop}[Teo]{Proposition}
\newtheorem{Lema}[Teo]{Lemma}

\theoremstyle{definition}
\newtheorem{Def}[Teo]{Definition}

\newcommand{\Q}{\mathbb{Q}}
\newcommand{\R}{\mathbb{R}}

\newcommand{\N}{\mathbb{N}}

\newcommand{\lra}{\longrightarrow}

\newcommand{\CH}{\mbox{\rm conv}}

\begin{document}
\title{Perron transforms and Hironaka's game}
\author{Michael de Moraes}
\author{Josnei Novacoski}
\thanks{During the realization of this project the second author was supported by a grant from Funda\c c\~ao de Amparo \`a Pesquisa do Estado de S\~ao Paulo (process number 2017/17835-9).}

\begin{abstract}
In this paper we present a matricial result that generalizes Hironaka's game and Perron transforms simultaneously. We also show how one can deduce the various forms in which the algorithm of Perron appears in proofs of local uniformization from our main result.
\end{abstract}

\maketitle

\section{Introduction}

The algorithm of Perron appears as an important tool in various proofs of local uniformization for valuations centered on algebraic varieties. For instance, in \cite{Zar1}, Zariski applies this algorithm in the proof of the Local Uniformization Theorem for places of algebraic function fields over base fields of characteristic $0$. Then, in \cite{Zar2}, he uses this theorem to prove resolution of
singularities for algebraic surfaces (in characteristic $0$). The local uniformization problem over base fields of any characteristic is still open. In \cite{Kuhl}, Knaf and Kuhlmann use a similar algorithm in the proof that Abhyankar places admit local uniformization in any characteristic. Also, in \cite{Cut}, Cutkosky and Mourtada use a version of Perron transforms in the proof that reduction of the multiplicity
of a characteristic $p>0$ hypersurface singularity along a valuation is possible if there is a finite linear projection which is defectless.

The Hironaka's game was proposed by Hironaka in \cite{Hirjog} and \cite{Hirjog1}. This game encodes the combinatorial part of the resolution a given singularity. Different winning strategies for this game allow different resolutions for that singularity. The existence of a winning strategy for Hironaka's game was first proved by Spivakovsky in \cite{Spiv}. An alternative solution was presented in \cite{Zei}. In \cite{Hauser}, Hauser presents a detailed relation between Hironaka's game and its applications on resolution of singularities.

The main goal of this paper is to explicitly relate Perron transforms and Hironaka's game. Our main result (Theorem \ref{mainresult}), which is given in terms of matrices with non-negative integer entries, implies the existence of a winning strategy for the Hironaka's game and also the existence of Perron transforms with some required properties.

This paper is divided as follows. In Section 2, we present and prove our main theorem. In Section 3, we present and prove Lemma 4.2 of \cite{Kuhl} (Theorem \ref{TeotransfPerrKK} below). Knaf and Kuhlmann use this result as an important step to prove that every Abhyankar valuation admits local uniformization. A proof of Theorem \ref{TeotransfPerrKK} can be found in \cite{Elliot}, but we show here that it follows easily from Theorem 2.1. In Section 4, we present and prove Lemma 4.1 of \cite{Cut} (Theorem \ref{TPcutmoutasa} below). In \cite{Cut}, the authors refer to a proof of it in \cite{Cut0}. That proof is based on the original algorithm of Zariski to prove local uniformization. Again, we show that Theorem \ref{TPcutmoutasa} follows from Theorem \ref{mainresult}. In section 5, we present the Hironaka's game (also known as Hironaka's polyhedra game), and deduce from Theorem 2.1 that it admits a wining strategy (Theorem \ref{mainresultsec5}).

\section{Main theorem}
 
Let $J$ be a subset of $\{1,\ldots,n\}$ and $j\in J$. We define the $n\times n$ matrix 
\[
A_{J,j}=(a_{rs})_{1\leq r,s\leq n}
\]
by
\begin{displaymath}
a_{rs}=\left\{ \begin{array}{ll}
1 &\mbox{ if }r=s\mbox{ or if }r=j\mbox{ and }s\in J\\
0&\mbox{ otherwise }\\
\end{array}.\right.
\end{displaymath}
Notice that $\det{A_{J,j}}=1$ and if we think of $A_{J,j}$ as a mapping from $\N^n$ to $\N^n$ we have
\[
A_{J,j}(\alpha_1,\ldots,\alpha_n)=(\alpha_1,\ldots,\alpha_{j-1},\sum_{i\in J}\alpha_i,\alpha_{j+1}\ldots,\alpha_n).
\]

The main result of this paper is the following:
\begin{Teo}\label{mainresult}
Let $\alpha,\beta\in\N^n$. Then there exist $l\in\mathbb{N}$, subsets $J_1,\ldots,J_l\subseteq\{1,\ldots,n\}$ and $j_1,\ldots,j_l$ such that for every $k$, $1\leq k\leq l$, $J_k$ is chosen in function of the set
\[\{\alpha,\beta,J_1,\ldots,J_{k-1},j_1,\ldots, j_{k-1}\},
\]
and $j_k$ is randomly assigned in $J_k$ such that
\[
A\alpha\leq A\beta\mbox{ or }A\beta\leq A\alpha
\]
componentwise, where
\[
A=A_{J_l,j_l}\ldots A_{J_1,j_1}.
\]
\end{Teo}

Given $\alpha=(a_1,\ldots,a_n),\beta=(b_1,\ldots,b_n)\in \N^n$, we define $\tau(\alpha,\beta)\in\mathbb{N}\times\mathbb{N}$ in the following way: set
\[
\gamma_{\alpha\beta}=(c_1,\ldots,c_n)\mbox{ where }c_i:=\min\{a_i,b_i\},
\]
and denote $\overline\alpha:=\alpha-\gamma_{\alpha\beta}$ and $\overline\beta:=\beta-\gamma_{\alpha\beta}$. Then
\[
\tau(\alpha,\beta)=(\min\{|\overline\alpha|,|\overline{\beta}|\},\max\{|\overline\alpha|,|\overline{\beta}|\})\in\mathbb{N}\times\mathbb{N},
\]
where $|$ $|$ denotes the sum norm. Observe that $\alpha\leq\beta$ or $\beta\leq\alpha$ componentwise if, and only if, the first coordinate of $\tau(\alpha,\beta)$ is $0$.

\begin{Prop}\label{mainprop}
Let $\alpha,\beta\in \N^n$ such that $0<|\overline\alpha|$ and $0<|\overline\beta|$, where $\overline{\alpha}=\alpha-\gamma_{\alpha\beta}$ and $\overline{\beta}=\beta-\gamma_{\alpha\beta}$. Then there exists $J\subseteq \{1,\ldots,n\}$ such that, for every $j\in J$, we have
\[
\tau(A_{J,j}\alpha,A_{J,j}\beta)<_{lex}\tau(\alpha,\beta),
\]
where $<_{lex}$ denotes the lexicographic order.
\end{Prop}

We proof now Theorem \ref{mainresult} using Proposition \ref{mainprop}, and we will prove Proposition \ref{mainprop} in the sequence.

\begin{proof}[Proof of Theorem \ref{mainresult} assuming Proposition \ref{mainprop}]
We set 
\[
\alpha^{(0)}:=\alpha\mbox{, }\beta^{(0)}:=\beta
\] 
and for $k\geq 1$, if $\alpha^{(k-1)}$, $\beta^{(k-1)}$, $J_k$ and $j_k$ have been defined, we set 
\[
\alpha^{(k)}:=A_{J_k,j_k}\alpha^{(k-1)}\mbox{, }\beta^{(k)}:=A_{J_k,j_k}\beta^{(k-1)}.
\]
We have to show that for some $l\in\mathbb{N}$, the first coordinate of $\tau(\alpha^{(l)},\beta^{(l)})$ is $0$, where we choose $J_k$, and $j_k$ is randomly assigned in $J_k$ for all $k\leq l$.

If the first coordinate of $\tau(\alpha,\beta)$ is $0$, then nothing needs to be done. Suppose that the first coordinate of $\tau(\alpha,\beta)$ is different than $0$. By Proposition \ref{mainprop}, there exists $J_1\subseteq\{1,\ldots,n\}$ such that for any $j_1\in J_1$ we have
\[
\tau(\alpha^{(1)},\beta^{(1)})<_{lex}\tau(\alpha,\beta).
\]
If $\tau(\alpha^{(1)},\beta^{(1)})$ has first coordinate $0$, it is done. If not, we apply proposition 1.2 again. Iterating this process, we produce a strictly descending sequence
\[
\tau(\alpha,\beta)>_{lex}\tau(\alpha^{(1)},\beta^{(1)})>_{lex}\tau(\alpha^{(2)},\beta^{(2)})>_{lex}\ldots.
\]
Since $\mathbb{N}\times\mathbb{N}$ is well ordered with respect to the lexicographic order, a strictly descending sequence must be finite. Then there is $l\in\mathbb{N}$ such that $\tau(\alpha^{(l)},\beta^{(l)})$ have the first coordinate equals to $0$. 
\end{proof}

Now we will prove Proposition \ref{mainprop}.

\begin{proof}[Proof of Proposition \ref{mainprop}]
Let $\alpha=(a_1,\ldots,a_n)$ and $\beta=(b_1,\ldots,b_n)$, and assume, without loss of generality, that $|\overline\alpha|\leq|\overline\beta|$. Since $|\overline\alpha|>0$ and $|\overline\beta|>0$, there exists $i,i'\in\{1,\ldots,n\}$ such that $a_i>b_i$ and $a_{i'}<b_{i'}$. Hence, renumbering the indexes, there exists $r\in \{1,\ldots,n-1\}$ such that $a_i>b_i$ for $i\leq r$ and $a_i\leq b_i$ for $i>r$. Then we have
\[
\overline\alpha=(\overline a_1,\ldots,\overline a_r,0,\ldots,0) \mbox{ and }\overline\beta=(0,\ldots,0,\overline b_{r+1},\ldots,\overline b_n),
\]
with $\overline a_i=a_i-b_i$, $1\leq i\leq r$ and $\overline{b_j}=b_j-a_j$, $r<j\leq n$. We also assume that $\overline b_{r+1},\ldots,\overline b_n$ are in descending order. The set $J$ and $j\in J$ will be chosen after this permutation, and then we can return to the original configuration with the inverse permutation.

Since $0<|\overline\alpha|\leq|\overline\beta|$, there exists $k\in\{r+1,\ldots,n\}$ such that 
\begin{equation}\label{s2}
|\overline\alpha|\leq\sum_{i=1}^k\overline b_i\mbox{ and }|\overline\alpha|>\sum_{i=1}^{k-1}\overline b_i.
\end{equation}
Take $J=\{1,\ldots,k\}$. For any fixed $j\in J$ we set 
\[
\alpha'=A_{J,j}\alpha\mbox{ and }\beta'=A_{J,j}\beta.
\]
Then
\[\alpha'=(a_1,\ldots,a'_j=\sum_{i=1}^{k}a_i,\ldots,a_n)\mbox{ and }
\beta'=(b_1,\ldots,b'_j=\sum_{i=1}^{k}b_i,\ldots,b_n).
\]

To calculate $\tau(\alpha',\beta')$, we denote 
\[\overline{\alpha'}=\alpha'-\gamma_{\alpha'\beta'}\mbox{ and }\overline{\beta'}=\beta'-\gamma_{\alpha'\beta'}.
\]
One can show that
\[
|\overline{\alpha'}|=|\overline\alpha|-\overline{a}_j+\overline{a'}_j\mbox{ and }|\overline{\beta'}|=|\overline\beta|-\overline{b}_j+\overline{b'}_j,
\]
where $\overline{a'}_j$ is the $j$th coordinate of $\overline{\alpha'}$, and $\overline{b'}_j$ is the $j$th coordinate of $\overline{\beta'}$.

We claim that $\overline{a'}_j=0$. Indeed, we know that $\overline{a'}_j=a'_j-\min\{a'_j,b'_j\}=\max\{a'_j-b'_j,0\}$, but
\[
a'_j-b'_j=\sum_{i=1}^{k}a_i-\sum_{i=1}^{k}b_i=\sum_{i=1}^r(a_i-b_i)-\sum_{i=r+1}^k(b_i-a_i)
\]
\[
=\sum_{i=1}^k\overline{a}_i-\sum_{i=1}^k\overline{b}_i=|\overline\alpha|-\sum_{i=1}^k\overline b_i\leq 0,
\]
because $|\overline\alpha|\leq\sum_{i=1}^k\overline b_i$.

We will analyze the cases $j\leq r$ and $j>r$ separately. 

If $j\leq r$, then $\overline a_j>0$, and we have
\[
|\overline{\alpha'}|=|\overline\alpha|-\overline a_j<|\overline\alpha|.
\]
Hence
\[
\tau(\alpha',\beta')=(\min\{|\overline{\alpha'}|,|\overline{\beta'}|\},\max\{|\overline{\alpha'}|,|\overline{\beta'}|\})\leq_{lex}(|\alpha'|,|\beta'|)<_{lex}(|\alpha|,|\beta|)=\tau(\alpha,\beta)
\]
and the result follows.

If $j>r$, we have $|\overline{\alpha'}|=|\overline\alpha|$, since $\overline{a'}_j=\overline{a}_j=0$. However, in this case, we claim that $\overline{b'}_j=\max\{b'_j-a'_j,0\}<\overline{b}_j$. Indeed, since $k\geq j>r$ and the $\overline b_i's$ are decreasing for $i>r$, the inequalities (\ref{s2}) guarantee that $
\overline b_j>0$. For the inequality $b_j'-a_j'<\overline{b}_j$, we will use that $\overline b_j\geq\overline b_k$, since the $\overline b_i's$ are decreasing for $i>r$ and $r<j\leq k$. We have
\[
b'_j-a'_j=\sum_{i=1}^k\overline b_i-|\overline\alpha|\leq\overline{b_j}+\sum_{i=1}^{k-1}\overline b_i-|\overline\alpha|.
\]
Since $|\overline{\alpha}|>\sum_{i=1}^{k-1}\overline{b}_i$, we have $\sum_{i=1}^{k-1}\overline{b}_i-|\overline{\alpha}|<0$, and then
\[
b'_j-a'_j\leq\overline{b_j}+\sum_{i=1}^{k-1}\overline b_i-|\overline\alpha|<\overline b_j.
\]
Finally,
\[
|\overline{\beta'}|=|\overline{\beta}|-\overline b_j+\overline{b'}_j<|\overline{\beta}|,
\]
and therefore
\[
\tau(\alpha',\beta')=(\min\{|\overline{\alpha'}|,|\overline{\beta'}|\},\max\{|\overline{\alpha'}|,|\overline{\beta'}|\})\leq_{lex}(|\alpha'|,|\beta'|)<_{lex}(|\alpha|,|\beta|)=\tau(\alpha,\beta).
\]
\end{proof}

\section{Kuhlmann and Knaf's Perron transform}

Let $\Gamma$ be a finitely generated ordered abelian group and $\mathcal{B}=\{\gamma_1,\ldots,\gamma_n\}$ a basis of $\Gamma$ (i.e., $\Gamma=\gamma_1\mathbb{Z}\oplus\ldots\oplus\gamma_n\mathbb{Z}$) formed by positive elements. Such basis exists because every ordered abelian group is free; see \cite{Elliot}.
\begin{Def}
A \textit{simple Perron transform} on $\mathcal{B}$ is a new basis $\mathcal{B}_1=\{\gamma^{(1)}_1,\ldots,\gamma^{(1)}_n\}$ of $\Gamma$, obtained in the following way: let $J\subseteq\{1,\ldots,n\}$ and $j\in J$ such that $\gamma_j\leq\gamma_i$ for all $i\in J$. Then 
\[
\gamma^{(1)}_i=\left\{
\begin{array}{ll}
\gamma_i-\gamma_j &\mbox{if }i\in J\setminus\{j\}\\
\gamma_i & \mbox{otherwise}
\end{array}.
\right.
\]
\end{Def}
Observe that $\mathcal{B}_1$ is indeed a basis of $\Gamma$ and is formed by positive elements, since $\gamma_i>\gamma_j$ for all $i\in J\setminus\{j\}$. We define a \textit{Perron transform} on $\mathcal{B}$ as a basis $\mathcal{B}'$, obtained by perform finitely many successive simple Perron transforms starting from $\mathcal{B}$. We denote, whenever necessary, 
\begin{equation}\label{s31}
\mathcal{B}'=(\mathcal{B}_0=\mathcal{B},\mathcal{B}_1,\ldots,\mathcal{B}_{l-1},\mathcal{B}_l=\mathcal{B}'),
\end{equation}
where $\mathcal{B}_k$ is a simple Perron transform of $\mathcal{B}_{k-1}$, for $k=1,\ldots,l$.

Let $\alpha\in\Gamma$. If $\alpha$ is written on the basis $\mathcal{B}$ by
\[
\left[\alpha\right]_\mathcal{B}=(a_1,\ldots,a_n)=a_1\gamma_1+\ldots+a_n\gamma_n,
\]
then $\alpha$ is written on the basis $\mathcal{B}_1$ by
\begin{equation}\label{s32}
\left[\alpha\right]_{\mathcal{B}_1}=(a_1,\ldots,a_{j-1},\sum_{i\in J}a_i,a_{j+1},\ldots,a_n)=\left[A_{J,j}\alpha\right]_\mathcal{B}.
\end{equation} 
The matrix $A_{J,j}$ is the matrix of change of basis, from $\mathcal{B}$ to $\mathcal{B}_1$, which we also denoted by $A_{\mathcal{B},\mathcal{B}_1}$. If $\mathcal{B}'$ is the Perron transform (\ref{s31}), then
\[
A_{\mathcal{B},\mathcal{B'}}=A_{\mathcal{B}_{l-1},\mathcal{B}'}\ldots A_{\mathcal{B},\mathcal{B}_1}.
\]

For a subset $\mathcal{D}$ of $\Gamma$, we denote
\[
\langle\mathcal{D}\rangle_+:=\left\{\displaystyle\sum_{i=1}^nm_id_i\ |\ m_i\in\mathbb{N},\ d_i\in\mathcal{D}\right\}.
\]
We see by (\ref{s32}) that 
\[
\langle\mathcal{B}\rangle_+\subseteq\langle\mathcal{B}_1\rangle_+\subseteq\ldots\subseteq\langle\mathcal{B}'\rangle_+.
\]

\begin{Lema}\label{Kulhmhdsjkad}
Let $\Gamma$ be a finitely generated ordered abelian group, $\mathcal{B}$ a basis of $\Gamma$ formed by positive elements and $\alpha\in\Gamma$ a positive element. Then there exists a Perron transform $\mathcal{B}'$ of $\mathcal{B}$ such that $\alpha\in\langle B'\rangle_+$.
\end{Lema}
\begin{proof}
Note that $\alpha\in\langle \mathcal{B}'\rangle_+$ if, and only if, $\alpha$ has non-negative coordinates on the basis $\mathcal{B}'$.

Write
\[
\alpha=\alpha_+-\alpha_-,
\]
were $[\alpha_+]_{\mathcal{B}}$ and $[\alpha_-]_{\mathcal{B}}$ have non-negative coordinates. By Theorem 2.1, there is a matrix
\[A=A_{J_l,j_l}\ldots A_{J_1,j_1}\]
such that $j_1,\ldots,j_l$ is given so that $A$ is the change-of-basis matrix of a Perron transform $\mathcal{B'}$ on $\mathcal{B}$, and
\[
A[\alpha_+]_\mathcal{B}\geq A[\alpha_-]_\mathcal{B}\mbox{ or }A[\alpha_+]_\mathcal{B}\leq A[\alpha_-]_\mathcal{B}
\]
componentwise. Since $\mathcal{B}'$ is formed by positive elements and $\alpha$ is positive, the equation
\[
[\alpha]_{\mathcal{B}'}=[\alpha_+]_{\mathcal{B}'}-[\alpha_-]_{\mathcal{B}'}
\]
ensures that
\[
[\alpha_+]_{\mathcal{B}'}=A[\alpha_+]_\mathcal{B}\geq A[\alpha_-]_\mathcal{B}=[\alpha_-]_{\mathcal{B}'}
\]
componentwise. Then
\[
[\alpha]_{\mathcal{B}'}=[\alpha_+]_{\mathcal{B}'}-[\alpha_-]_{\mathcal{B}'}
\]
has non-negative coordinates, and therefore
\[
\alpha\in\langle\mathcal{B}'\rangle_+.
\]
\end{proof}

In \cite{Kuhl} Knaf and Kuhlmann use the following result as an important step to prove that every Abhyankar valuation admits local uniformization.

\begin{Teo}[Lemma 4.2 of \cite{Kuhl}]\label{TeotransfPerrKK}
Let $\Gamma$ be a finitely generated ordered abelian group and $\alpha_1,\ldots, \alpha_l\in\Gamma$ positive elements. Then there exists a basis $\mathcal{B}$ of $\Gamma$, formed by positive elements, such that
\[
\alpha_1,\ldots,\alpha_l\in\langle\mathcal{B}\rangle_+.
\]
\end{Teo}

\begin{proof} We start with a basis $\mathcal{B}_0$ of $\Gamma$ formed by positive elements. By Lemma 3.1, there is a Perron transform $\mathcal{B}_1$ of $\mathcal{B}_0$ such that
\[
\alpha_1\in\langle\mathcal{B}_1\rangle_+.
\]
By Lemma \ref{Kulhmhdsjkad}, there is a Perron transform $\mathcal{B}_2$ of $\mathcal{B}_1$ such that
\[
\alpha_2\in\langle\mathcal{B}_2\rangle_+.
\]
Since $\langle\mathcal{B}_1\rangle_+\subseteq\langle\mathcal{B}_2\rangle_+$, we have
\[
\alpha_1,\alpha_2\in\langle\mathcal{B}_2\rangle_+.
\]
Repeating this process to include all elements $\alpha_1,\ldots,\alpha_l$, take $\mathcal{B}=\mathcal{B}_l$, and we have
\[
\alpha_1,\ldots,\alpha_l\in\langle\mathcal{B}\rangle_+.
\]
\end{proof}

\section{Cutkosky and Mourtada's Perron transform}

Let $k[x_1,\ldots,x_m]$ be a polynomial ring over a field $k$. Let $\nu$ be a valuation in $k[x_1,\ldots,x_m]$ with center $(x_1,\ldots,x_m)$, that is, $\nu(k^\times)=0$ and $\nu(f)>0$ for all $f\in(x_1,\ldots,x_n)$. Suppose that $\mathcal{B}=\{\nu(x_1),\ldots,\nu(x_n)\}$ is a rational basis of $\Gamma_\nu\otimes\Q$, where $\Gamma_\nu$ is the value group of $\nu$. Let $x_1',\ldots, x_m'$ be such that
\begin{displaymath}
x_i=\left\{
\begin{array}{ll}
\displaystyle\prod_{j=1}^n\left(x_j'\right)^{a_{ij}} &\mbox{if }1\leq i\leq n\\
x_i' &\mbox{if }n<i\leq m
\end{array},
\right.
\end{displaymath}
where $a_{ij}\in \N$, $\det(a_{ij})=1$ and $0<\nu\left(x_i'\right)$ for every $1\leq i\leq m$.
In \cite{Cut}, the authors call the inclusion map $k[x_1,\ldots,x_m]\lra k\left[x_1',\ldots,x_m'\right]$ a Perron transform of type (6).

\begin{Teo}[Lemma 4.1 of \cite{Cut}]\label{TPcutmoutasa}
Let $M_1=x_1^{d_1}\ldots x_n^{d_n}$ and $M_2=x_1^{e_1}\ldots x_n^{e_n}$ two monomials, with $d_1,\ldots,d_n,e_1,\ldots,e_n\geq 0$ and $\nu(M_1)<\nu(M_2)$. Then there exists a Perron transform of type (6) such that $M_1$ divides $M_2$ in $k[x_1',\ldots, x_m']$.
\end{Teo}

\begin{proof}
Since $\mathcal{B}$ is a rational basis of $\Gamma_\nu\otimes\mathbb{Q}$, we have that $\mathcal{B}$ is a basis, formed by positive elements, of the ordered subgroup $\Gamma:=\nu(x_1)\mathbb{Z}\oplus\ldots\oplus\nu(x_n)\mathbb{Z}$ of $\Gamma_v$. By Lemma \ref{Kulhmhdsjkad}, since $\nu(M_2)-\nu(M_1)>0$, there is a Perron transform $\mathcal{B}'=\{\gamma_1',\ldots,\gamma_n'\}$ such that $\nu(M_2)-\nu(M_1)\in\langle\mathcal{B}'\rangle_+$. 

Since $\langle\mathcal{B}\rangle_+\subseteq\langle\mathcal{B}'\rangle_+$, we have, for all $i=1,\ldots,n$,
\[
\nu(x_i)=\sum_{j=1}^n a_{ij}\gamma_j',
\]
where $a_{ij}\in\mathbb{N}$ for all $i,j=1,\ldots,n$. We have that $(a_{ij})$ is a matrix with non-negative entries and det$(a_{ij})=1$, since it represents a change of basis of a Perron transform.

Define $x_1',\ldots,x_m'$ by the equations
\begin{displaymath}
x_i=\left\{
\begin{array}{ll}
\displaystyle\prod_{j=1}^n\left(x_j'\right)^{a_{ij}} &\mbox{if }1\leq i\leq n\\
x_i' &\mbox{if }n<i\leq m
\end{array}.
\right.
\end{displaymath}
We have $\nu(x_j')=\gamma_j'>0$ for every $j$, $1\leq j\leq n$. Furthermore,
\[
\nu(M_2M_1^{-1})=\nu(M_2)-\nu(M_1)=b_1\gamma_1'+\ldots+b_n\gamma_n',
\]
where $b_1,\ldots,b_n\in\mathbb{N}$, since $\nu(M_2)-\nu(M_1)\in\langle\mathcal{B}'\rangle_+$. Then
\[
M_2M_1^{-1}={x'_1}^{b_1}\ldots {x'_n}^{b_n}\in k[x_1',\ldots,x_n'],
\]
and therefore $M_1$ divides $M_2$ in $k[x_1',\ldots,x_n']$.
\end{proof}

If $\mathcal{B}'$ is a simple Perron transform, then the first $n\times n$ quadrant of the matrix $(a_{ij})$ above is the matrix $A_{\mathcal{B},\mathcal{B}'}$, 
which is of the form $A_{J,j}$, for some $J\subseteq\{1,\ldots,n\}$ and $j\in J$. Then the $x'_i$'s are defined by 
\begin{displaymath}
x_i'=\left\{
\begin{array}{ll}
\frac{x_i}{x_j} &\mbox{if } i\in J\setminus\{j\}\\
x_i & otherwise
\end{array}.
\right.
\end{displaymath}
If $\mathcal{B}'$ is a Perron transform, then the $x_i$'s are defined by iteration of the above definition to simple Perron transforms.

As a corollary of Theorem 4.1 we have the following: 

\begin{Teo}
Let $f\in k[x_1,\ldots,x_m]$ be a polynomial. Then there is a Perron transform of type (6) such that $f$ is written by
\[
f={x_1'}^{b_1}\ldots{x_n'}^{b_n}g,
\]
where $b_1,\ldots,b_n\in\mathbb{N}$ and $g\in k[x_1',\ldots,x_m']\setminus (x_1',\ldots,x_n')$.
\end{Teo}

\begin{proof}
Let $R=k[x_{n+1},\ldots,x_m]$. Write $f$ as
\[
f=g_1M_1+\ldots+g_kM_k,
\]
where $g_i\in R\setminus\{0\}$ and $M_i$ are monomials of $R[x_1,\ldots,x_n]$ with coefficients equal to $1$, for all $i=1,\ldots,k$. We may assume that $\nu(M_1)<\nu(M_i)$ for all $i=2,\ldots,k$. 
Applying Theorem 4.1 successively, there exists a type $(6)$ Perron transform such that $M_1$ divides $M_i$ in $k[x_1',\ldots, x_m']$ for all $i=1,\ldots,k$. Then
\[
f=M_1\left(g_1+g_2\frac{M_2}{M_1}+\ldots g_k\frac{M_k}{M_1}\right),
\]
where
\[
g_1+g_2\frac{M_2}{M_1}+\ldots g_k\frac{M_k}{M_1}\in k[x_1',\ldots,x_m']\setminus (x_1',\ldots,x_n').
\]
\end{proof}

\section{Hironaka's game}

Let $V$ be a finite number of points in $\N^n$ with positive convex hull $N:=\CH(V+\R_{\geq 0}^n)$. Consider two ``players", $\mathcal P_1$ and $\mathcal P_2$, competing in the following game (known as Hironaka's polyhedra game): player $\mathcal P_1$ chooses a subset $J$ of $\{1,\ldots,n\}$ and, afterwards, player $\mathcal P_2$ chooses and element $j\in J$. After this ``round", the set $V$ is replaced by the set $V_1$ obtained as follows: for each element $\alpha=(a_1,\ldots,a_n)\in V$ the corresponding element $\alpha^1=(b_1,\ldots,b_n)\in V_1$ will be
\[
b_j:=\sum_{i\in J}a_i\mbox{ e }b_k:=a_k\mbox{ se }k\neq j.
\]
We define then $N_1:=\CH(V_1+\R_{\geq 0}^n)$. Player $\mathcal P_1$ wins the game if, after finitely many rounds, the set $N_l$ becomes an ``ortant", i.e., a set of the form $\alpha+\R_{\geq 0}^n$ for some $\alpha\in\N^n$. The main result of \cite{Spiv} is the following:
\begin{Teo}\label{Hirgamespiv}
There exists a winning strategy for player $\mathcal P_1$.
\end{Teo}

Observe that $\alpha^1$ above is equal to $A_{J,j}\alpha$, and player $\mathcal{P}_1$ wins the game after the $l$th round if, and only if, the set $V_l$ has minimum element which respect to the componentwise order. Then we can refrase Theorem 5.1 in the following equivalent theorem:

\begin{Teo}\label{mainresultsec5}
Let $V\subset\N^n$ be a finite non-empty set. Then there exist $l\in\mathbb{N}$ and finite sequences $J_1,\ldots,J_l\subseteq\{1,\ldots,n\}$ and $j_1,\ldots,j_l$ such that $J_k$ is chosen in function of the set
\[
V\cup\{J_1,\ldots,J_{k-1},j_1,\ldots,j_{k-1}\},
\]
and 
\[
j_k\mbox{ is randomly assigned in }J_k,
\]
such that there exists $\alpha\in V$ for which
\[
A\alpha\leq A\beta\mbox{ componentwise, for every }\beta\in V,
\]
where
\[
A=A_{J_l,j_l}\ldots A_{J_1,j_1}.
\]
\end{Teo}

\begin{proof}
We prove by induction on $|V|$. If $|V|=1$ there is nothing to prove. The case $|V|=2$ is just Theorem \ref{mainresult} and it was already proved in Section 2. Suppose, by induction, that that the Theorem is true for set $V'$ with $|V'|=m$. Let $V=\{\alpha_1,\ldots,\alpha_m,\beta\}$. Applying the induction hypothesis to $V':=\{\alpha_1,\ldots,\alpha_m\}$, we may assume that $\alpha_1\leq\alpha_k$ componentwise for all $k=1,\ldots,m$. Since a matrix of the type $A_{J,j}$ preserves componentwise inequalities, we apply the case $|V|=2$ to $\{\alpha_1,\beta\}$ and the theorem is proved.
\end{proof}

\ \\
\noindent{\footnotesize JOSNEI NOVACOSKI\\
Departamento de Matem\'atica--UFSCar\\
Rodovia Washington Lu\'is, 235\\
13565-905 - S\~ao Carlos - SP\\
Email: {\tt josnei@dm.ufscar.br} \\\\

\noindent{\footnotesize MICHAEL DE MORAES\\
Departamento de Matem\'atica--ICMC-USP\\
Av. Trabalhador s\~ao-carlense, 400\\
13566-590 - S\~ao Carlos - SP\\
Email: {\tt michael.moraes@usp.br} \\\\
\end{document}